\documentclass[oneside]{amsart}
\usepackage{amsmath,amssymb,amsfonts,amsthm}
\usepackage{color}
\usepackage{graphicx}
\usepackage{geometry} 
\usepackage{amscd}
\usepackage{enumitem}

\title{\Large A \MakeLowercase{new look  at}  F\MakeLowercase{insler surfaces and the} L\MakeLowercase{andsberg's} PDE  }

\author[Elgendi]{\bf S. G. ~E\MakeLowercase{lgendi}}

\address{Salah G. Elgendi, Department of Mathematics, Faculty of Science, Benha
	University, Egypt} \email{salah.ali@fsci.bu.edu.eg, \, salahelgendi@yahoo.com}
\urladdr{http://www.bu.edu.eg/staff/salahali7}



\thanks{}

\def\blue#1{\textcolor[rgb]{0.0,0.0,1.0}{#1}}

\newcommand{\T}{{\mathcal T}}
\newcommand{\C}{{\mathcal C}}

\newcommand{\Real}{\mathbb R}

\newcommand{\abs}[1]{\left\vert#1\right\vert}

\newcommand{\To}{\longrightarrow}

\newcommand{\tm}{\T M}

\def\H{{{\mathcal D}_{\mathcal H}}}

\def\pa{\partial}
\def\paa{\dot{\partial}}

\def\+{\!+\!}

\def\={\!=\!}
\def\<{\!<\!}
\def\>{\!>\!}

 \let\oldmarginpar\marginpar
\renewcommand\marginpar[1]{\oldmarginpar[\raggedleft\footnotesize #1]%
  {\blue{\raggedright \footnotesize \fbox{
      \begin{minipage}{1.0\linewidth}
        #1
      \end{minipage}
}}}}

\numberwithin{equation}{section} 
\numberwithin{figure}{section} 

\theoremstyle{plain}
\newtheorem*{theorem*}{Theorem}
\newtheorem{theorem}{Theorem}[section]
\newtheorem{lemma}[theorem]{Lemma}
\newtheorem{proposition}[theorem]{Proposition}
\newtheorem{corollary}[theorem]{Corollary}
\theoremstyle{definition}
\newtheorem{definition}[theorem]{Definition}

\theoremstyle{remark}
\newtheorem{example}{Example}
\newtheorem{remark}[theorem]{Remark}
\newtheorem*{acknowledgement*}{Acknowledgement}

\begin{document}

\maketitle

\bigskip


\bigskip

\begin{abstract}
 In this paper, we introduce a new look at Finsler surfaces. Landsberg surfaces are Finsler surfaces that are solutions of a system of non-linear partial differential equations.  Considering the unicorn's Landsberg problem,  we reduce this system to a single non-linear PDE which we call the Landsberg's PDE. By making use of the new look of Finsler surfaces, we solve the Landsberg's PDE and get a class of solutions. Moreover, we show  that these solutions and their  conformal transformations   are Berwaldain.
 \end{abstract}

\section{Introduction}

A Finsler manifold $(M,F)$ is said be  \textit{Berwald } if the coefficients of  Berwald connection depend only on the position arguments and this is equivalent to that the Berwald parallel translation $P_c$ along a curve $c(t)$ is linear isometry  between $(T_pM,F_p)$  and $(T_qM,F_q)$ where $c(t)$ joins the points $p,q\in M$.
 A Finsler manifold $(M,F)$  is said be \textit{Landsberg } if the horizontal covariant derivative of the metric tensor of $F$ with respect to the Berwald connection vanishes and this is equivalent to the fact that the   parallel translation $P_c$ along $c$ preserves the induced Riemannian metrics on the slit tangent spaces, i.e., $P_c : (T_pM \backslash \{0\},{g}_p) \longrightarrow (T_qM \backslash \{0\}, {g}_q)$ is an isometry. There are many other characterizations for  Berwald and Landsberg metrics.

 It is known that every Berwald space is  Landsberg, but the converse  is a long-existing problem
in Finsler geometry, which is still open. Matsumoto,  one of the best geometers who had a very   significant contribution  to Finsler geometry in the last century,  called the problem  the most important unsolved problem in Finsler geometry. Many geometers in the Finslerian area of research  looking for a regular Landsberg space which is not Berwald. Even they want to know if such spaces exists or not.   From the applicable point of view, especially in Physics,  G. Asanov \cite{Asanov} obtained class of metrics (singular or y-local) arising from Finslerian General Relativity. These metrics are    non-Berwaldian Landsberg metrics. Later, Z. Shen \cite{Shen_example}, generalized this class and classified all Landsberg  $(\alpha,\beta)$-metrics which are not Berwaldian.
Due to the many unsuccessful attempts made for finding non-Berwaldian Landsberg metrics, D. Bao \cite{Bao} called them \lq\lq unicorns”.

\medskip

 There are some papers devoted to the unicorn problem in dimension two. For example, R. Bryant claimed there  exists the singular Landsberg Finsler surfaces which are not Berwaldian, moreover among them there is  surfaces  with vanishing flag curvature (cf. \cite{Bao}). Later, Zhou \cite{Zhou} confirmed R. Bryant's claim by giving examples of Landsberg surfaces which are not Berwaldain.  Recently in \cite{Zhou-note} it was shown that the examples obtained by Zhou are in fact Berwaldian. By the way, the spherically symmetric Landsberg surfaces are studied in \cite{Elgendi-SSM}. In \cite{Thompson}, G. Thompson claimed that there are Landsberg Finsler metrics in dimension two which are not Berwaldian, but no  concrete  examples of such surfaces are given. Also, one can see the work of S. V. Sabau  in dimension two, for example, see \cite{Sorin}. For concrete examples and further studies of the unicorns for higher dimensions we refer to \cite{Elgendi-LBp,Elgendi-solutions,Elgendi-ST_condition}.

\medskip

In this paper, we rewrite the Finsler function on any two-dimensional manifold as follows:
\begin{equation*}
  F=\abs{y^1}f(x,\varepsilon u),      \quad u=\frac{y^2}{y^1}, \  \ \varepsilon:=\operatorname{sgn}(y^1)
\end{equation*}
where $f(x,\varepsilon u):=F(x,\varepsilon ,\varepsilon u)$ is a positive smooth function on $M\times\Real$   and $|\cdot |$ is the absolute value.
It should be noted that if we start by regular Finsler function $F$, then the Finsler function $F(x,y)=\abs{y^1} f(x,\varepsilon u)$ is regular although the function $u$ has a singularity at $y^1=0$.
As an  example  (cf. \cite[Example 1.2.2 Page 15]{shen-book1}):
$$F(x,y)=\sqrt{(y^1)^2+(y^2)^2}+B y^1=|y^1|\left(\sqrt{1+u^2}+\varepsilon B\right).$$
In this example $f(x, \varepsilon u)=\sqrt{1+u^2}+\varepsilon B$.

\medskip

We calculate the Berwald and Landsberg curvatures. We show that the Landsberg condition leads to a single  PDE. By solving this PDE we show that all  two-dimensional   Landsberg metrics on the form
\begin{equation*}
	F=	\abs{y^1}\phi(t), \quad t:=\rho(x^1,x^2) u, \quad u :=y^2/y^1
\end{equation*}
 and their conformal transformations are Berwaldian (cf. Theorems \ref{Theorem_A} and \ref{Theorem_B}).
As by-product, we get  explicit formulae for these solutions, precisely, the classes of the Landsberg solutions given by
\begin{equation*}
	F=	\sqrt{c_3\rho^2(y^2)^2-(c-2)\rho y^1y^2-2c_1 (y^1)^2}  \ e^{ \frac{c}{\sqrt{c^2-4c_2}}\ \operatorname{arctanh}\left( \frac{2c_3 \rho y^2-(c-2)y^1}{y^1\sqrt{c^2-4c_2}}\right) },
\end{equation*}
are Berwaldian where  $c:=2c_1c_3+c_2+1$ and $c_2>0, a,b,c_1,c_3$ are constants (cf. Theorem \ref{Theorem_C}).

In case of higher dimensions $n\geq 3$, the conformal transformation of a Minkowski metrics can yield a non-Berwaldian Landsberg metrics (cf. \cite{Elgendi-LBp}). Assuming that a manifold $(M,F)$ is Minkowski if the Finsler function $F$ depends on the directional argument $y$ only, in contrast of the higher dimensions, we prove that if the conformal transformation  of any two dimensional  Minkowski metric is Landsberg then it must be  Berwaldian (cf. Theorem \ref{Theorem_D}).

\section{Preliminaries}

Let $M$ be an $n$-dimensional manifold and $(TM,\pi_M,M)$ be its tangent bundle
and $(\T M,\pi,M)$ the subbundle of nonzero tangent vectors.  We denote by
$(x^i) $ local coordinates on the base manifold $M$ and by $(x^i, y^i)$ the
induced coordinates on $TM$.  The vector $1$-form $J$ on $TM$ defined,
locally, by $J = \frac{\partial}{\partial y^i} \otimes dx^i$ is called the
natural almost-tangent structure of $T M$. The vertical vector field
$\C=y^i\frac{\partial}{\partial y^i}$ on $TM$ is called the canonical or the
Liouville vector field.

A vector field $S\in \mathfrak{X}(\T M)$ is called a spray if $JS = \C$ and
$[\C, S] = S$. Locally, a spray can be expressed as follows
\begin{equation*}
  \label{eq:spray}
  S = y^i \frac{\partial}{\partial x^i} - 2G^i\frac{\partial}{\partial y^i},
\end{equation*}
where the \emph{spray coefficients} $G^i=G^i(x,y)$ are $2$-homogeneous
functions in the $y=(y^1, \dots , y^n)$ variable.

A nonlinear connection is defined by an $n$-dimensional distribution $H : u \in \tm \rightarrow H_u\subset T_u(\tm)$ that is supplementary to the vertical distribution, which means that for all $u \in \tm$, we have
\begin{equation}
\label{eq:direct_sum}
T_u(\tm) = H_u(\tm) \oplus V_u(\tm).
\end{equation}

Every spray S induces a canonical nonlinear connection through the corresponding horizontal and vertical projectors,
\begin{equation*}
  \label{projectors}
    h=\frac{1}{2}  (Id + [J,S]), \,\,\,\,            v=\frac{1}{2}(Id - [J,S])
\end{equation*}
Equivalently, the canonical nonlinear connection induced by a spray can be expressed in terms of an almost product structure  $\Gamma = [J,S] = h - v$. With respect to the induced nonlinear connection, a spray $S$ is horizontal, which means that $S = hS$. Locally, the two projectors $h$ and $v$ can be expressed as follows
$$h=\frac{\delta}{\delta x^i}\otimes dx^i, \quad\quad v=\frac{\partial}{\partial y^i}\otimes \delta y^i,$$
$$\frac{\delta}{\delta x^i}=\frac{\partial}{\partial x^i}-G^j_i(x,y)\frac{\partial}{\partial y^j},\quad \delta y^i=dy^i+G^j_i(x,y)dx^i, \quad G^j_i(x,y)=\frac{\partial G^j}{\partial y^i}.$$
Moreover, the coefficients of the Berwald connection is given by
$$ G^h_{ij}=\frac{\partial G^h_j}{\partial y^i}.$$

\begin{displaymath}
  R=\frac{1}{2}[h,h]=\frac{1}{2}R^i_{jk}\frac{\partial}{\partial
    y^i}\otimes dx^j \wedge dx^k, \qquad R^i_{jk} =
  \frac{\delta
    G^i_j}{\delta x^k} - \frac{\delta G^i_k}{\delta x^j}
\end{displaymath}
is called the curvature of $S$. From the curvature tensor one can obtain
the Riemann curvature \cite{shen-book1} (or the Jacobi endomorphism, see
\cite{Bucataru1}), which is defined by
\begin{equation}
  \label{eq:16}
  \Phi= R^i_j \,  dx^j \otimes \frac{\partial}{\partial y^i},
  \qquad   R^i_j =   2\frac{\partial G^i}{\partial x^j} - S(G^i_j) -
  G^i_k G^k_j.
\end{equation}
The two curvature tensors are related by
\begin{equation}
  \label{eq:R_Phi}
  \Phi=i_SR, \qquad 3R=[J,\Phi],
\end{equation}
respectively.

From now on,  for simplicity,  we use the notations
$$\pa_i:=\frac{\partial}{\partial x^i}, \quad \dot{\partial}_i:=\frac{\partial}{\partial y^i}.$$

\begin{definition}
A Finsler manifold  of dimension $n$ is a pair $(M,F)$, where $M$ is a  differentiable manifold of dimension $n$ and $F$ is a map  $$F: TM \To \Real ,\vspace{-0.1cm}$$  such that{\em:}
 \begin{description}
    \item[(a)] $F$ is smooth and strictly positive on $\T M$ and $F(x,y)=0$ if and only if $y=0$,
    \item[(b)]$F$ is positively homogenous of degree $1$ in the directional argument $y${\em:}
    $\mathcal{L}_{\mathcal{C}} F=F$,
    \item[(c)] The metric tensor $g_{ij}= \dot{\partial}_i \dot{\partial}_j E$ has rank $n$ on $\T M$, where $E:=\frac{1}{2}F^2$ is the energy function.
 \end{description}
 \end{definition}

 Since the $2$-form $dd_JE$ is
non-degenerate,  the Euler-Lagrange equation
\begin{equation*}
  \label{eq:EL}
  i_Sdd_JE=-dE
\end{equation*}
uniquely determines a spray $S$ on $TM$.  This spray is called the
\emph{geodesic spray} of the Finsler function.

\begin{definition}
  A spray $S$ on a manifold $M$ is called \emph{Finsler metrizable} if there
  exists a Finsler function $F$ such that the geodesic spray of the Finsler
  manifold $(M,F)$ is $S$.
  \end{definition}

  It is known that a spray $S$ is Finsler metrizable if and only if there exists a non-degenerate  solution $F$    for the system
  \begin{equation}
  \label{metrizable_system}
  d_hF=0, \quad d_\C F=F,
  \end{equation}
  where $h$ is the horizontal projector associated to $S$.

The local formula for the hv-curvature tensor $G$ of Berwlad connection and the Landsbeg tensor $L$ given, respectively,  by
\begin{equation}
\label{Berwald_curv.}
G=G^h_{ijk} dx^i\otimes dx^j\otimes dx^k\otimes\paa_h
\end{equation}
\begin{equation}
\label{Landsberg_Tensor}
L=L_{ijk} dx^i\otimes dx^j\otimes dx^k,
\end{equation}
where $G^h_{ijk}:=\paa_k G^h_{ij}$ and $L_{ijk}:=-\frac{1}{2}F G^h_{ijk}\ell_h$.
 \begin{definition}
A Finsler manifold $(M,F)$ is said to be \textit{Berwald} if and only if $G^{h}_{ijk}$ vanishes identically.
\end{definition}

  \begin{definition}
A Finsler manifold $(M,F)$  is called \textit{Landsberg} if if and only if $L_{ijk}$ vanishes identically.
\end{definition}

\section{A new look at Finsler surfaces and the Landsberg's PDE}

In what follows,  $\pa_1$ (rep. $\pa_2$) stands for the partial differentiation with respect to $x^1$ (resp. $x^2$) and $\paa_1$ (resp. $\paa_2$) stands for the partial differentiation with respect to $y^1$ (resp. $y^2$).

\medskip

The following lemmas are useful for subsequent use.
\begin{lemma}\label{Lemma_G1_G2}
  Let $F$ be a Finsler function on a two-dimensional manifold $M$, then $F$ can be  written  in the form
\begin{equation}\label{Finsler_Function}
  F=\abs{y^1}f(x,\varepsilon u),      \quad u=\frac{y^2}{y^1}, \  \ \varepsilon:=\operatorname{sgn}(y^1)
\end{equation}
where $f(x,\varepsilon u):=F(x,\varepsilon ,\varepsilon u)$ is a positive smooth function on $M\times\Real$   and $|\cdot |$ is the absolute value.
Moreover, for the expression $ F=\abs{y^1}f(x,\varepsilon u)$,  the  coefficients $G^1$ and $G^2$ of the geodesic spray are  given by
\begin{equation}\label{G^1_G^2}
  G^1=f_1(x,u)(y^1)^2, \quad  G^2=f_2(x,u)(y^1)^2,
\end{equation}
where  the functions $f_1$ and $f_2$ are smooth functions on $M\times\Real$ and  given as follows
\begin{equation}\label{G_f1}
f_1=\frac{(\pa_1f+u\pa_2f)f''-(\pa_1f'+u\pa_2f'-\pa_2f)f'}{2ff''},
\end{equation}
\begin{equation}\label{G_f2}
f_2=\frac{u(\pa_1f+u\pa_2f)f''+(\pa_1f'+u\pa_2f'-\pa_2f)( f- uf')}{2ff''},
\end{equation}
where    $f'$ (resp. $f''$) is the first (resp. the second) derivative of  $f$  with respect to $ u$ and so on.
\end{lemma}

Before we go to prove the above  lemma, let's give the following remark.
\begin{remark}
 It should be noted that if we start by regular Finsler function $F$, then the Finsler function $F(x,y)=\abs{y^1} f(x,\varepsilon u)$ is regular although the function $u$ has a singularity at $y^1=0$.
As an  example  (cf. \cite[Example 1.2.2 Page 15]{shen-book1}):
$$F(x,y)=\sqrt{(y^1)^2+(y^2)^2}+B y^1=|y^1|\left(\sqrt{1+u^2}+\varepsilon B\right).$$
In this example $f(x, \varepsilon u)=\sqrt{1+u^2}+\varepsilon B$.
\end{remark}
\begin{proof}[Proof of Lemma \ref{Lemma_G1_G2}]
Since $F=0 $ only on the zero section, then away from the zero section at each $x\in M$, at least one   of the $y$'s is non zero. Without loss of generality, we assume that $y^1\neq 0$. Then by using the fact that the Finsler function $F$ is positively homogeneous of degree $1$ in $y$, we have
$$F(x,y^1,y^2)=F\left(x, \abs{y^1} \frac{y^1}{\abs{y^1}}, \abs{y^1}\frac{y^2}{\abs{y^1}}\right)=\abs{y^1}F(x,\varepsilon,\varepsilon u)=\abs{y^1} f(x,\varepsilon u),$$
where $f(x,\varepsilon u):=F(x,\varepsilon,\varepsilon u)$.

 Now, assume that   $ F=\abs{y^1}f(x^1,x^2,\varepsilon u)$. Since the coefficients $G^i$ of the geodesic spray of $F$ are homogeneous of degree $2$ in $y$, then we can write
 $$ G^1=f_1(x^1,x^2,u)(y^1)^2, \quad  G^2=f_2(x^1,x^2,u)(y^1)^2.$$
  Keeping the facts that $\displaystyle{\frac{\partial f}{\partial (\varepsilon u)}}=\varepsilon \frac{\partial f}{\partial u}=\varepsilon f'$ and $\varepsilon^2=1$ in mind, then straightforward calculations lead to
  $$\ell_1:=\paa_1F=\varepsilon f+ \abs{y^1}f' (-  y^2/(y^1)^2) =\varepsilon f-\varepsilon  uf'=  \varepsilon( f-uf'),$$
$$  \ell_2:=\paa_2F=\abs{y^1} f' (1 /y^1)=  \frac{\abs{y^1}}{y^1} f'=\varepsilon f'.$$
  The coefficients $N^i_j$ of the non-linear connection can calculated in the form
    \begin{equation}\label{Nonlinear_connection}
  \begin{split}
       N^1_1=\paa_1G^1=2y^1f_1-y^2f_1', \quad N^1_2=\paa_2G^1=y^1f_1',  \\
       N^2_1=\paa_1G^2=2y^1f_2-y^2f_2',\quad N^2_2=\paa_2G^2=y^1f_2'. \\
\end{split}
  \end{equation}
   Now, we have to  determine the coefficients $G^1$ and $G^2$ of the geodesic spray of $F$. Since $S$ is the geodesic spray of $F$, then $d_hF=0$ that is $\pa_i F-N^h_i\paa_h F=0$. Then we have the system
    $$\abs{y^1}\pa_1f-(2y^1f_1-y^2f_1')(\varepsilon f-\varepsilon uf')-(2y^1f_2-y^2f_2')\varepsilon f'=0,$$
    $$\abs{y^1}\pa_2f-y^1f_1'(\varepsilon f-\varepsilon uf')-y^1f_2' \varepsilon f'=0.$$
Since we assume that $y^1\neq 0$, then dividing the above system by $\abs{y^1}$, we get
  \begin{equation}\label{Eq:G1_1}
   \pa_1f-2f_1(  f-  uf')+uf_1'( f-   uf')-2  f_2f'+  uf_2'f'=0,
  \end{equation}
  \begin{equation}\label{Eq:G2_1}
  \pa_2f- f_1'(f-  uf')- f_2'f'=0.
  \end{equation}
  Multiplying \eqref{Eq:G2_1} by $u$ and then adding it to \eqref{Eq:G1_1}, we have
  \begin{equation}\label{Eq:spray_1}
   \pa_1f+u\pa_2f-2f_1(f- uf')-2 f_2f'=0.
  \end{equation}
  Differentiating \eqref{Eq:G1_1} and \eqref{Eq:G2_1} with respect to $u$ yields the following
   \begin{equation}\label{Eq:G1_2}
     \pa_1f'+2uf_1f''-f_1'(f- uf')+u(f_1'(f- uf'))'- f_2'f'-2f_2f''+ u(f_2'f')'=0
  \end{equation}
  \begin{equation}\label{Eq:G2_2}
  \pa_2f'-(f_1'(f- uf'))'-(f_2'f')'=0.
  \end{equation}
  Multiplying \eqref{Eq:G2_2} by $u$ and then adding it to \eqref{Eq:G1_2} and using \eqref{Eq:G2_1}, we get
   \begin{equation}\label{Eq:spray_2}
    \pa_1f'+ u\pa_2f'-\pa_2f+2uf_1f''-2f_2f''=0.
  \end{equation}
  Multiplying \eqref{Eq:spray_1} by $f''$ and \eqref{Eq:spray_2} by $ f'$ and then by subtraction, we obtain the required  formula for $f_1$. Finally, substituting by $f_1$ into \eqref{Eq:spray_2}, we get the formula of $f_2$.
\end{proof}

\begin{lemma}\label{Berwald_Connection}
The coefficients $G^h_{jk}$ of Berwald connection are given by
  \begin{equation}\label{Berwald_connection}
  \begin{split}
       G^1_{11}&=2f_1-2uf_1'+u^2f_1'', \quad  G^1_{12}=f_1'-uf_1'', \quad  G^1_{22}=f_1'',  \\
        G^2_{11}&=2f_2-2uf_2'+u^2f_2'', \quad  G^2_{12}=f_2'-uf_2'', \quad  G^2_{22}=f_2''. \\
\end{split}
  \end{equation}
\end{lemma}
\begin{proof}
The proof comes from  the fact that $G^h_{jk}=\paa_kN^h_j$ together with \eqref{Nonlinear_connection}.
\end{proof}

\begin{proposition}
The components   $R^i_{jk}$ of the curvature tensor are given by
\begin{equation}
\label{Eq:R_1}
R^1_{12}=y^1(uf_1'^2-f_1'f_2'-2uf_1f_1''+2f_2f_1''-u\pa_2f_1' + 2 \pa_2f_1-\pa_1f_1'),
\end{equation}
\begin{equation}
\label{Eq:R_2}
R^1_{12}=y^1(-2uf_1f_2''+2f_2f_2''+uf_1'f_2'-2f_2f_1'+2f_1f_2'-f_2'^2-u\pa_2f_2'+2\pa_2f_2-\pa_1f'_2).
\end{equation}
 \end{proposition}
\begin{proof}
Using the the fact that $R^h_{jk}=\delta_kN^h_j-\delta_kN^h_j$, we have
$$R^1_{12}=\delta_2N^1_1-\delta_1N^1_2=\pa_2N^1_1-N^1_2\paa_1N^1_1-N^2_2\paa_2N^1_1-\pa_1N^1_2+N^1_1\paa_1 N^1_2+N^2_1\paa_2N^1_2,$$

$$R^2_{12}=\delta_2N^2_1-\delta_1N^2_2=\pa_2N^2_1-N^1_2\paa_1N^2_1-N^2_2\paa_2N^2_1-\pa_1N^2_2+N^1_1\paa_1 N^2_2+N^2_1\paa_2N^2_2.$$
Now by substituting from   \eqref{Nonlinear_connection} and \eqref{Berwald_Connection} into the above formulae we get the result.
\end{proof}

\begin{example}
Let $M=\Real^2$ and  $F$ be a Finsler function given by
$$F=\sqrt{e^{x^2}(y^1)^2+(y^2)^2}.$$
Then we have
 $$f(x,u)=\sqrt{e^{x^2} +u^2}.$$
 By substituting into \eqref{G_f1} and \eqref{G_f2}, we get  $f_1=\frac{1}{2}u$,  $f_2=-\frac{1}{4}e^{x^2}$.
 Hence, the spray coefficients are given by
 $$G^1=(y^1)^2\frac{1}{2}u=\frac{1}{2}y^1y^2, \quad G^2=-\frac{1}{4}e^{x^2}(y^1)^2=-\frac{1}{4}e^{x^2} (y^1)^2.$$
 Making use of \eqref{Eq:R_1}, \eqref{Eq:R_2}, we have
 $$R^1_{12}=y^1\left(\frac{1}{4}u\right)=\frac{1}{4}y^2,\quad R^2_{12}=-y^1\left(\frac{1}{4}e^{x^2}\right)=-\frac{1}{4}e^{x^2} y^1.$$
\end{example}

\medskip

Now, for any surface $(M,F)$,  we can calculate the components of the Berwald and Landsberg curvatures as follows.

\begin{lemma}
  The components $G^h_{ijk}$ of Berwald curvature are given by
  \begin{equation}\label{Berwald_tensors}
  \begin{split}
       G^1_{111}&=-\frac{u^3}{y^1}f_1''', \quad  G^2_{111}=-\frac{u^3}{y^1}f_2''', \quad  G^1_{211}=\frac{u^2}{y^1}f_1''', \quad  G^2_{211}=\frac{u^2}{y^1}f_2''', \\
        G^1_{221}&=-\frac{u}{y^1}f_1''', \quad  G^2_{221}=-\frac{u}{y^1}f_2''', \quad  G^1_{222}=\frac{1}{y^1}f_1''', \quad  G^2_{222}=\frac{1}{y^1}f_2'''. \\
\end{split}
  \end{equation}
  The components $L_{ijk}$ of the Landsberg curvature are given by
  \begin{equation}\label{Lands_tensors}
  \begin{split}
       L_{111}&=\frac{u^3f}{2}(f_1'''\ell_1+f_2'''\ell_2), \quad  L_{112}=-\frac{u^2f}{2}(f_1'''\ell_1+f_2'''\ell_2),  \\
        L_{122}&=-\frac{uf}{2}(f_1'''\ell_1+f_2'''\ell_2), \quad  L_{222}=-\frac{f}{2}(f_1'''\ell_1+f_2'''\ell_2).
\end{split}
  \end{equation}
\end{lemma}
\begin{proof}
 We  do the calculations for one component and the rest can be calculated  in a similar manner. By using \eqref{Berwald_Connection} and the definition of the Berwald tensor $G^h_{ijk}=\paa_kG^h_{ij}$,  we can calculate, for example,  $G^1_{111}$
as follows
 $$G^1_{111}=\paa_1G^1_{11}=(2f_1-2uf_1'+u^2f_1'')'\frac{-y^2}{(y^1)^2}=-\frac{u^3}{y^1}f_1'''.$$
 By the same way, one   gets the component $G^2_{111}=-\frac{u^3}{y^1}f_2'''$. Hence,  the first component $L_{111}$ of the Landsberg curvature is given by
 $$L_{111}=-\frac{F}{2}(G^1_{111}\ell_1+G^2_{111}\ell_2)=\frac{u^3f}{2}(f_1'''\ell_1+f_2'''\ell_2).$$
\end{proof}

The above lemma gives rise to the following proposition.
\begin{proposition}\label{Landsberg_PDE}
  Any  two dimensional Finsler  manifold $(M,F)$ in the form \eqref{Finsler_Function} is  Landsbergian if and only if the following PDE
  \begin{equation}
  \label{Eq:Landsberg_PDE}
  f_1'''\ell_1+f_2'''\ell_2=0
  \end{equation}
is satisfied. The above PDE will be called the two-dimensional Landsberg's PDE.
\end{proposition}
\begin{proof}
  Since in the Landsberg space all the components of the Landsberg tensor $L_{ijk}$ vanishes identically, then by making use of \eqref{Lands_tensors} the result follows.
\end{proof}

\section{Solutions to the Landsberg's  PDE}

We are going to  solve the Landsberg's PDE in a special case, so  we focus our attention to the Finsler functions on the form
\begin{equation}\label{Eq:Main_Form}
F=\abs{y^1}\phi(t), \quad t:=\rho(x^1,x^2) u,
\end{equation}
where we consider, without loss of generality, that $y^1\neq 0$.

\begin{remark}
  It should be noted that one of our goals is to solve the Landsberg's PDE \eqref{Eq:Landsberg_PDE}, so when $\phi(t)=t$ in \eqref{Eq:Main_Form} then we use the separation of variables method to solve \eqref{Eq:Landsberg_PDE}. Therefore, the  assumption \eqref{Eq:Main_Form} is a more general method  than the separation of variables.
\end{remark}

Let's define  the function $Q$ as follows
$$Q:=\frac{\phi_t}{\phi-t\phi_t},$$
where the subscript $t$ refers to the differentiation with respect to $t$.  Moreover, in this case, $\phi$ is given by
\begin{equation}\label{Q_f}
 \phi=\exp\left(\int \frac{Q}{1+tQ}dt\right).
\end{equation}

\begin{lemma}\label{Lemma_Q}
The spray coefficients of the Finsler functions given in \eqref{Eq:Main_Form}, are determined  by the functions
\begin{equation}\label{Eq:f1}
  f_1=-\frac{Q^2}{Q_t} \frac{\pa_1\rho}{2\rho},
\end{equation}
\begin{equation}\label{Eq:f2}
  f_2=\frac{ u\pa_1\rho+u^2 \pa_2\rho }{2\rho}+\frac{Q}{Q_t} \frac{\pa_1\rho}{2\rho^2}.
\end{equation}
\end{lemma}
\begin{proof}
    Let $F$ be given by \eqref{Eq:Main_Form} and  $f(x,u)=\phi(t)$. We have the following formulas
$$f'=  \rho \ \phi_t, \quad f''=  \rho^2 \ \phi_{tt},$$
$$\partial_1f=    u \ \phi_t \ \partial_1\rho, \quad \partial_2f=  u \ \phi_t \ \partial_2\rho , $$
$$\partial_1f'=   \phi_t \ \partial_1\rho +  \rho \ u \ \phi_{tt} \ \partial_1\rho, \quad \partial_2f'=   \phi_t \ \partial_2\rho+ \rho \  u  \ \phi_{tt} \ \partial_2\rho.$$
 Substituting by the above equalities into \eqref{G_f1} and \eqref{G_f2} and  by making use of the facts that $Q=\frac{  \phi_{t}}{\phi-t\phi_t}$ and  $Q_t=\frac{\phi \phi_{tt}}{(\phi-t\phi_t)^2}$,   we get the required formulae for $f_1$ and $f_2$.
\end{proof}

Focusing more on $Q$, we have the following lemma.
\begin{lemma}\label{Special_Q}
  Let the Finsler function be given in the form   $F=\abs{y^1}  \phi( t)$. Then, we have the following assertions
  \begin{itemize}
    \item[(a)] If $Q_{tt}=0$, then the manifold $(M,F)$ is Berwaldian.
    \item[(b)] If $Q=\theta(x)$, then the Finsler function is degenerate.
  \end{itemize}
\end{lemma}
\begin{proof}
To prove (a), if $Q_{tt}=0$, then $Q=at+b$, where $a$ and $b$ are arbitrary constants. Moreover, we have the following
$$\frac{Q^2}{Q_t}=\frac{a^2t^2+2abt+b^2}{a}, \quad \frac{Q}{Q_t}=\frac{at+b}{a}.$$
Making use of the above quantities together with  \eqref{Eq:f1} and \eqref{Eq:f2}, we conclude that the coefficients of the spray are quadratic and hence the metric is Berwaldain.

To prove (b), one can see that, if $Q=\theta(x)$, then we have
$$\frac{Q}{1+tQ}=\frac{\theta(x)}{1+t\theta(x)}.$$
Therefore, by using \eqref{Finsler_Function} and \eqref{Q_f}, we have
$$F=\abs{y^1}\exp(\ln (1+t\theta(x))= \varepsilon (y^1+\rho\theta(x)y^2).$$
This means that the Finsler function is linear and hence the metric tensor is degenerate.
\end{proof}

Now, we are in a position to announce the main   result in this work.
\begin{theorem}\label{Theorem_A}
All  two-dimensional   Landsberg metrics in the form \eqref{Eq:Main_Form}, that is,
\begin{equation*}
F=\abs{y^1}f(x^1,x^2,u), \quad f(x^1,x^2,u)=  \phi(t), \quad t:=\rho(x^1,x^2) u, \quad u =y^2/y^1, \quad y^1\neq 0.
\end{equation*}
  are Berwaldian.
\end{theorem}
\begin{proof}
Let $(M,F)$ be two-dimensional Landsbergian, then by Proposition \ref{Landsberg_PDE}, the Landsberg's PDE is satisfied, that is
  $$f_1'''\ell_1+f_2'''\ell_2=0.$$
  Now, assume that  the solutions of the Landsberg's PDE  is in  the form
 $$
  f(x^1,x^2,u)=\phi(t), \quad t=\rho(x) u.
 $$

  Since $\ell_1=\varepsilon( f-uf')=\varepsilon( \phi-t\phi_t)$ and $\ell_2=\varepsilon f'=\varepsilon \rho \phi_t$, the above PDE can be rewritten in the form
  \begin{equation}\label{PDE_basic}
  f_1'''+\rho Qf_2'''=0.
  \end{equation}

  Now, we are going to find analytic solution for \eqref{PDE_basic}.  By making use of  \eqref{Eq:f1} and \eqref{Eq:f2}, we have
    $$f_1'''=\frac{Q\pa_1\rho}{2\rho Q_t^4}\left( QQ_t^2Q_{tttt}-6QQ_tQ_{tt}Q_{ttt}+6QQ_{ttt}^3+4Q_t^3Q_{ttt}-6Q_t^2Q_{tt}^2 \right),$$
  $$f_2'''=-\frac{\pa_1\rho}{2\rho^2 Q_t^4}\left( QQ_t^2Q_{tttt}-6QQ_tQ_{tt}Q_{ttt}+6QQ_{ttt}^3+2Q_t^3Q_{ttt}-3Q_t^2Q_{tt}^2 \right).$$
  Now, using  Lemma \ref{Special_Q} (a) together with the facts that the functions $Q$, $\rho$, $\pa_1\rho$ and $Q_t$ should  be non-zero,   the PDE \eqref{PDE_basic} transformed to  the ODE
    $$3Q_{tt}^2-2Q_tQ_{ttt}=0.$$
  Making use of Lemma \ref{Special_Q} (b), $Q_{tt}\neq 0$. Hence, the above ODE can be rewritten in the form
  $$1+2\left( \frac{Q_t}{Q_{tt}}\right)_t=0.$$
  Moreover, the above ODE has the solution
  $$\frac{Q_t}{Q_{tt}}=-\frac{1}{2}t+c_1,$$
  where $c_1$ is arbitrary constant. Furthermore, we can find $Q_t$, since
  $$\frac{Q_{tt}}{Q_t}=\frac{2}{2c_1-t}.$$
  Which gives easily the formula of $Q_t$ as follows
   \begin{equation}\label{Q_t}
   Q_t=\frac{c_2}{(2c_1-t)^2}, \quad c_2>0.
   \end{equation}
  By solving the above differential equation, we get

   \begin{equation}\label{Q}
   Q=\frac{c_2}{2c_1-t}+c_3,
  \end{equation}
  where $c_1,c_2,c_3$ are arbitrary constants such that $c_2>0$.

  Now, substituting by   $Q$ and $Q_t$ into $f_1$ and $f_2$ given in   Lemma \ref{Lemma_Q} and then substituting into the coefficients $G^1 $ and $G^2$ given by   \eqref{G^1_G^2},  we conclude that the coefficients of the spray are quadratic and hence the metric is Berwaldain.
  \end{proof}

Consider  the conformal transformation of the metrics on the form \eqref{Eq:Main_Form}, precisely,
\begin{equation}\label{Conformal_change}
  \overline{F}=\sigma(x^1,x^2) F=\abs{y^1}\sigma \phi(t)
\end{equation}

Now, to get information about the transformed spray (the geodesic spray of $\overline{F}$),  we have to calculate $\overline{f}_1$ and $\overline{f}_2$.

\begin{lemma}\label{f_1&2_bar}
Under the conformal transformation \eqref{Conformal_change}, the functions $f_1$ and $f_2$ transform as follows:
\begin{eqnarray*}
   \overline{f}_1 = f_1+  \frac{\partial_1\sigma+u\partial_2\sigma}{2\sigma}+\frac{\partial_2\sigma}{2\sigma \rho}\frac{Q}{Q_t}- \frac{\partial_1\sigma}{2\sigma } \frac{Q^2}{Q_t},
\end{eqnarray*}

\begin{eqnarray*}
   \overline{f}_2=  f_2+  \frac{u(\partial_1\sigma+u\partial_2\sigma)}{2\sigma}+\frac{\partial_1\sigma}{2\rho\sigma}\frac{Q}{Q_t}-\frac{\partial_2\sigma}{2\sigma \rho^2}\frac{1}{Q_t}.
\end{eqnarray*}
\end{lemma}
\begin{proof}
  Consider  the conformal transformation \eqref{Conformal_change}, precisely,

$$\overline{F}=\sigma(x^1,x^2) F=\abs{y^1}\sigma \phi(t), \quad t =\rho(x^1,x^2) u, \quad u =y^2/y^1.$$

Since $\overline{F}=\abs{y^1} \overline{f}(x,u)=\abs{y^1}\sigma(x^1,x^2) f(x,u)$, where
$$\overline{f}=\sigma f.$$

Now,  we have to calculate $\overline{f}_1$ and $\overline{f}_2$. For  this purpose, we have the following

$$\overline{f}'=\sigma f', \quad \overline{f}''=\sigma f'',$$
$$\partial_1\overline{f}=\sigma \partial_1f+f \partial_1\sigma  , \quad \partial_2\overline{f}=\sigma \partial_2f+f\partial_2\sigma,  $$
$$\partial_1\overline{f}'=\sigma \partial_1f'+f' \partial_1\sigma  , \quad \partial_2\overline{f}'=\sigma \partial_2f'+f'\partial_2\sigma.$$

 By  making use of the above relations together with the help of the quantities $Q=\frac{  \phi_{t}}{\phi-t\phi_t}$,  $Q_t=\frac{\phi \phi_{tt}}{(\phi-t\phi_t)^2}$, $f'=\rho \phi_t$ and $f''=\rho^2 \phi_{tt}$, then  \eqref{G_f1} and \eqref{G_f2} lead to
\begin{eqnarray*}
   \overline{f}_1 &=&  f_1+  \frac{\partial_1\sigma+u\partial_2\sigma}{2\sigma}+\frac{\partial_2\sigma}{2\sigma \rho}\frac{Q}{Q_t}- \frac{\partial_1\sigma}{2\sigma } \frac{Q^2}{Q_t},
\end{eqnarray*}

\begin{eqnarray*}
   \overline{f}_2&=&  f_2+  \frac{u(\partial_1\sigma+u\partial_2\sigma)}{2\sigma}+\frac{\partial_1\sigma}{2\rho\sigma}\frac{Q}{Q_t}-\frac{\partial_2\sigma}{2\sigma \rho^2}\frac{1}{Q_t}.
\end{eqnarray*}
As to be shown.
\end{proof}

\begin{theorem}\label{Theorem_B}
The conformal transformations of the Finsler functions given in \eqref{Eq:Main_Form} by any positive smooth function $\sigma(x^1,x^2)$ are Berwaldain.
\end{theorem}

\begin{proof}
  By the formulae \eqref{Q_t} and \eqref{Q}, we have

  $$ Q_t=\frac{c_2}{(2c_1-t)^2},\quad   Q=\frac{c_2}{2c_1-t}+c_3.$$
  Taking  the fact that the geodesic spray of the class \eqref{Eq:Main_Form}  is quadratic into account. Then,  substituting by $Q$ and $Q_t$ into the formulae of $\overline{f}_1$ and $\overline{f}_2$ which are given in Lemma \ref{f_1&2_bar},  we conclude that the geodesic spray of $\overline{F}$ is quadratic and hence the space is Berwaldian.
\end{proof}

As by-product, we get the following explicit formulae for two dimensional   Landsberg  metrics which are Berwaldain.
\begin{theorem}\label{Theorem_C}
 The Landsberg metric in the form \eqref{Eq:Main_Form}  is given by
\begin{equation} \label{Eq:Phi_1}
F=	\sqrt{c_3\rho^2(y^2)^2-(c-2)\rho y^1y^2-2c_1 (y^1)^2}  \ e^{ \frac{c}{\sqrt{c^2-4c_2}}\ \operatorname{arctanh}\left( \frac{2c_3 \rho y^2-(c-2)y^1}{y^1\sqrt{c^2-4c_2}}\right) }.
\end{equation}
 As a special case, let $\rho(x)=1$ so that    $\phi(t)=\phi(u)$.  Then by making use of Theorem \ref{Theorem_B}, we have
 $$f(x^1,x^2,u)= \sigma(x^1,x^2) \phi(u)$$
 and we get the class
 \begin{equation} \label{Eq:Phi_special}
 F=\sigma(x)\sqrt{c_3(y^2)^2-(c-2) y^1y^2-2c_1 (y^1)^2} \ e^{ \frac{c}{\sqrt{c^2-4c_2}}\ \operatorname{arctanh}\left( \frac{2c_3  y^2-(c-2)y^1}{y^1\sqrt{c^2-4c_2}}\right) }.
 \end{equation}
   In addition, the metrics that satisfy that $Q_{tt}=0$ are  given by
    \begin{equation}\label{Eq:Phi_2}
   	F=\sqrt{a\rho^2(y^2)^2+b\rho y^1y^2+ (y^1)^2} \ e^{ - \frac{b}{\sqrt{b^2-4a}}\ \text{arctanh}\left( \frac{2a \rho y^2+by^1}{y^1\sqrt{b^2-4a}}\right) }.
   \end{equation}
   Where $t =\rho(x^1,x^2) u$, $u =\frac{y^2}{y^1}$, $c:=2c_1c_3+c_2+1$ and $c_2>0, a,b,c_1,c_3$ are constants.
\end{theorem}

\begin{proof}
To calculate the first class, by \eqref{Q},  one can show that
 $$\frac{Q}{1+tQ}=\frac{-c_3t+2c_1c_3+c_2}{-c_3t^2+(2c_1c_3+c_2-1)t+2c_1}$$
 which can be rewritten in the following useful form
  $$\frac{Q}{1+tQ}=\frac{1}{2}\frac{2c_3t-c+2}{c_3t^2-(c-2)t-2c_1}+\frac{2c c_3}{(c^2-4c_2)-(2c_3t-c+2)^2}.$$
  Hence, we have
    $$\int \frac{Q}{1+tQ} dt=\frac{1}{2}\ln{(c_3t^2-(c-2)t-2c_1)}+\frac{ c  }{\sqrt{c^2-4c_2}} \operatorname{arctanh}\left( \frac{2c_3 t-(c-2) }{ \sqrt{c^2-4c_2}}\right) .$$

   Now, using \eqref{Q_f}, the function $\phi(t)$ can be calculated and so we get the Finsler function.

To find the third class of  the Finsler functions, since $Q_{tt}=0$, then $Q=at+b$. Now, we have
$$\frac{Q}{1+tQ}=\frac{at+b}{at^2+bt+1}=\frac{2at+b}{2(at^2+bt+1)}-\frac{2ab}{b^2-4a-(2at+b)^2}.$$
Hence,  by the help of \eqref{Finsler_Function} and  \eqref{Q_f}, the required formula can be obtained.
\end{proof}

As a direct consequence, we can make use of  the surprising result due to Z. Szabo \cite{Szabo_Berwald} and, for more details, we refer to \cite{Szilasi-book}.
\begin{corollary}
All  two-dimensional   Landsberg metrics on the form  \eqref{Eq:Main_Form} and their conformal transformations  are flat (Locally Minkowski)or Riemannian.
\end{corollary}

In case of higher dimensions $n\geq 3$, the conformal transformation of a Minkowski metrics can yield a non-Berwaldian Landsberg metrics (cf. \cite{Elgendi-LBp}). Assume that a manifold $(M,F)$ is Minkowski if the Finsler function $F$ depends on the directional argument $y$ only. In contrast of the higher dimensions, we have the following result.

\begin{theorem}\label{Theorem_D}
If the conformal transformation  of any two dimensional  Minkowski metric is Landsberg, then it is  Berwaldian.
\end{theorem}
\begin{proof}
Let $(M,F)$ be a Minkowski space, then one can write $$F=\abs{y^1}f(u).$$
In this case $F$ can be considered a special case of the form \eqref{Eq:Main_Form} by setting $\rho(x)=1$.
Now, the conformal transformation of $F$ takes the form
$$\overline{F}=\sigma(x) F=\abs{y^1}\sigma(x)f(u).$$
Since  $\overline{F}$ is Landsbergian then $f(u)$ is given by  \eqref{Eq:Phi_special}, that is $f(u)=\phi(u)$. Consequently, $\overline{F}$ is Berwaldain.
\end{proof}

\section{Concluding remarks}

We end this work by the following    remarks.

$\bullet$ If we consider the Landsberg surface in the form $$F=\abs{y^2} \varphi(\rho(x) v), \quad v:=y^1/y^2,$$
then we get  a similar  formula of $\varphi(\rho(x)v)$ as the formula giving by \eqref{Eq:Phi_1}. In fact this corresponds  a change of coordinates by switching $x^1$ and $x^2$ and  this implies switching $y^1$ and $y^2$.  Moreover, if we need to fix a coordinate system and work on  $F=\abs{y^2} \varphi(\rho(x) v)$, then $u=\frac{1}{v}$. To get the function $\varphi$,   we have to substitute by $u=\frac{1}{v}$ in \eqref{Eq:Phi_1}, that is $\varphi(\rho(x) v)=\phi(\frac{\rho(x)}{v})$.  That is, we obtain the same formula of the Finsler function $F$.

$\bullet$ If we consider the Landsberg surface in the form $$F=\abs{y^2} \varphi(\rho(x) v), \quad v:=y^1/y^2,$$
then we get  a similar  formula of $\varphi(\rho(x)v)$ as the formula giving by \eqref{Eq:Phi_1}.

\medskip

$\bullet$ If there is a singularity in the surfaces given by \eqref{Eq:Phi_1} or \eqref{Eq:Phi_2} at certain direction, then this singularity does not come from the  assumption that $F=\abs{y^1}  \phi(\rho(x) u)$ rather it comes from being $F$  a solution  of the Landsberg's PDE. Moreover, the choice of the constants $c_1$, $c_2$, $c_3$ play an essential rule to determine whether the resultant metric is regular. Many, but not all, Berwaldian surfaces among them the Riemannian ones can be obtained by specific choice of these constants. For example the choice $c_1=c_2=1$ and  $c_3=-1$ leads to $c=0$ and hence we get  the Riemannian metrics
$$F=  \sqrt{ a(x)(y^1)^2+b(x)y^1y^2+c(x)(y^2)^2},$$
for some functions $a(x)$, $b(x)$, $c(x)$.

\providecommand{\bysame}{\leavevmode\hbox
to3em{\hrulefill}\thinspace}
\providecommand{\MR}{\relax\ifhmode\unskip\space\fi MR }
\providecommand{\MRhref}[2]{%
  \href{http://www.ams.org/mathscinet-getitem?mr=#1}{#2}
} \providecommand{\href}[2]{#2}

\end{document}